\documentclass[reqno,12pt]{amsart}
\usepackage{amssymb,latexsym}
\usepackage{array,multirow,makecell}
\usepackage{mathtools}
\usepackage{IEEEtrantools}
\usepackage{amsmath}
\setcellgapes{1pt}
\makegapedcells
\newtheorem{theorem}{Theorem}[section]

\newtheorem{lemma}[theorem]{Lemma}
\newtheorem{proposition}[theorem]{Proposition}
\newtheorem{korselt's criterion}[theorem]{Korselt's criterion}
\newtheorem{definition}[theorem]{Definition}

\newtheorem{example}[theorem]{Example}

\numberwithin{equation}{section}
\begin{document}
\title[Korselt Rational Bases and Sets ]
{Korselt Rational Bases and Sets}

\author{Nejib Ghanmi}
\address[Ghanmi]{(1)Preparatory Institute of Engineering Studies, Tunis university, Tunisia.}
\address[]{\hspace{1.5cm}(2)   University College of Jammum, Department of Mathematics, Mekkah, Saudi Arabia.}

\email{neghanmi@yahoo.fr \; and\; naghanmi@uqu.edu.sa }

\thanks{}



\begin{abstract} For  a positive integer $N$ and $\mathbb{A}$ a subset of $\mathbb{Q}$, let  $\mathbb{A}$-$\mathcal{KS}(N)$ denote the set of $\alpha=\dfrac{\alpha_{1}}{\alpha_{2}}\in \mathbb{A}\setminus \{0,N\}$   verifying $\alpha_{2}p-\alpha_{1}$ divides $\alpha_{2}N-\alpha_{1}$ for every prime divisor $p$ of $N$. The set  $\mathbb{A}$-$\mathcal{KS}(N)$ is   called the set of Korselt bases of $N$ in $\mathbb{A}$ or simply the $\mathbb{A}$-Korselt set of $N$.

In this paper, we prove that for each squarefree composite number $N\in\mathbb{N}\setminus\{0,1\}$ the $\mathbb{Q}$-Korselt set of $N$ is finite where we provide an upper and lower bounds for each Korselt base of $N$ in $\mathbb{Q}$. Furthermore, we give a necessary and a sufficient condition for the upper  bound  of a Korselt base to be reached.

\end{abstract}

\maketitle

\section{Introduction}

On $1640$ Fermat wrote a letter to Fernicle,  stating his famous assertion  now-called "Fermat Little Theorem":
 \begin{theorem}[Fermat Little Theorem]
 If $p$ is prime then $p$ divides $ a^{p}-a$  for all $a\in\mathbb{N}$.

 \end{theorem}

 The question arose whether the converse is true; the first answer was given by Carmichael~\cite{Car1, Car2} in $1910$ by showing  that $561$ is composite and verifies the converse of Fermat Little Theorem. Thus, all counterexamples to the converse of Fermat Little Theorem bear the name of Carmichael and defined as follows:
 \begin{definition}
A Carmichael number is a composite number $N$ such that $N$ divides $a^{N}-a$ for all $a\in\mathbb{N}$.
\end{definition}

The search of Carmichael was aided by an important criterion given  by A.Korselt~\cite{Kor} in $1899$ where these numbers are well characterized by a necessary and sufficient condition as  follows:

\begin{korselt's criterion}
A composite integer  $N>1$ is a Carmichael number if and only if $p-1$ divides $N-1$ for all prime factors  $p$ of $N$ .

\end{korselt's criterion}

This criterion simplified the study of Carmichael numbers  and helped to discover  the infinitude of Carmichael numbers in $1994$ by Alford-Granville-Pomerance~\cite{Ann}.
In the proof of the infinitude  of  Carmichael numbers, the authors asked if this proof can be generalized to produce another kind of pseudoprimes;
An  important response to this  question was given by Bouallegue-Echi-Pinch~\cite{BouEchPin}. In their work, the authors generalized the idea of Korselt by introducing the notion of  $\alpha$-Korselt numbers for $\alpha \in \mathbb{Z}$ as follows:
\begin{definition}
An $\alpha$-Korselt number is a number $N$ such that $p-\alpha$ divides $N-\alpha$ for all $p$ prime divisor of $N$.
\end{definition}
 Carmichael numbers are exactly the 1-Korselt  squarefree composite numbers. The $\alpha$-Korselt numbers for $\alpha \in \mathbb{Z}$ are well investigated last years specially in ~\cite{ BouEchPin, KorGhan, WilGhan, Ghanmi}. Motivated by these facts, Ghanmi ~\cite{Ghanmi2} introduced the  notion of $\mathbb{Q}$-Korselt numbers as extension  of the Korselt numbers to $\mathbb{Q}$ by setting the following definitions.

\begin{definition} Let $N\in \mathbb{N}\setminus\{0,1\}$,  $\alpha=\dfrac{\alpha_{1}}{\alpha_{2}}\in \mathbb{Q}\setminus \{0\}$ with $gcd(\alpha_{1}, \alpha_{2})=1$ and $\mathbb{A}$ a subset of $\mathbb{Q}$.
Then
\begin{enumerate}
  \item $N$ is said to be an \emph{$\alpha$-Korselt number\index{Korselt number}} (\emph{$K_{\alpha}$-number}, for
short), if $N\neq \alpha$ and $\alpha_{2}p-\alpha_{1}$ divides $\alpha_{2}N-\alpha_{1}$ for
every prime divisor $p$ of $N$.

\item By the \emph{$\mathbb{A}$-Korselt set}\index{Korselt set} of the number $N$ (or the Korselt set of $N$ over \emph{$\mathbb{A}$}), we mean the set $\mathbb{A}$-$\mathcal{KS}(N)$ of
all $\beta\in \mathbb{A}\setminus\{0,N\}$ such that $N$ is a $K_{\beta}$-number.
  \item The cardinality of $\mathbb{A}$-$\mathcal{KS}(N)$ will be called the \emph{$\mathbb{A}$-Korselt
weight}\index{Korselt weight} of $N$; we denote it by $\mathbb{A}$-$\mathcal{KW}(N)$.

\end{enumerate}

\end{definition}
   Further, in ~\cite{Ghanmi3} the autor  state  the following  definitions:

\begin{definition} Let $N\in \mathbb{N}\setminus\{0,1\}$, $\alpha\in \mathbb{Q}$ and $\mathbb{B}$ be a subset of $\mathbb{N}$. Then

\begin{enumerate}
\item  $\alpha$ is called \emph{$N$-Korselt base\index{Korselt base}}(\emph{$K_{N}$-base}, for
short), if $N$ is a \emph{$K_{\alpha}$-number}.
\item By the \emph{$\mathbb{B}$-Korselt set}\index{Korselt base set} of the base $\alpha$ (or the Korselt set of the base $\alpha$ over \emph{$\mathbb{B}$}), we mean the set $\mathbb{B}$-$\mathcal{KS}(B(\alpha))$ of
all $M\in \mathbb{B}$ such that $\alpha$ is a
$K_{M}$-base.
\item The cardinality of  $\mathbb{B}$-$\mathcal{KS}(B(\alpha))$ will be called the \emph{$\mathbb{B}$-Korselt
weight}\index{Korselt  weight} of the base $\alpha$; we denote it by $\mathbb{B}$-$\mathcal{KW}(B(\alpha))$.
\end{enumerate}
\end{definition}

For more convenience,  the set $\bigcup\limits_{N \in\mathbb{N}} ( \mathbb{Q} \text{-}\mathcal{KS}(N))$ is called the set of Korselt rational bases or the set of $\mathbb{N}$-Korselt bases in $\mathbb{Q}$ or  the set of  Korselt rational bases  over $\mathbb{N}$.

In this paper we are concerned only with a squarefree composite number $N$.

 Passing from $\mathbb{Z}$ to $\mathbb{Q}$, the Korselt set of a  number $N$ can vary widely, unlike other numbers for which the Korselt set remains unchanged. For instance,  if $N=2*11$ then $\mathbb{Z}$-$\mathcal{KS}(N)=\mathbb{Q}$-$\mathcal{KS}(N)=\{12\}$, but  for $N=71*73$, we have $\mathbb{Z}$-$\mathcal{KW}(N)=9$ and $ \mathbb{Q}$-$\mathcal{KW}(N)=285$. However, this does not prevent us from showing in section $2$ that  for each squarefree composite number $N$ there exist only  finitely many  $N$-Korselt rational bases. Moreover, we provide an upper and lower bounds for each $N$-Korselt rational base where we discuss the case when an upper bound is attained.

\section{Korselt rational base Properties}
 In the whole section and for $\alpha=\dfrac{\alpha_{1}}{\alpha_{2}}\in\mathbb{Q}$, we will suppose, without loss of generality, that $\alpha_{2}>0$, $\alpha_{1}\in \mathbb{Z}$ and $\gcd(\alpha_{1}, \alpha_{2})=1$.

Further,  for  $N=p_{1}p_{2}\ldots p_{m}$  with $p_{1}< p_{2}<\ldots < p_{m}$, we set

$(\mathbb{Q}_{+})$-$\mathcal{KS}(N)=\{\gamma_{1},\ldots,\gamma_{r} ; \ \ 0<\gamma_{1}<\ldots<\gamma_{r} \}$ and

  $(\mathbb{Q}_{-})$-$\mathcal{KS}(N)=\{\beta_{1},\ldots , \beta_{t}; \ \ \beta_{1}<\ldots < \beta_{t}<0\}$

  whenever $(\mathbb{Q}_{+})$-$\mathcal{KS}(N)$ and $(\mathbb{Q}_{-})$-$\mathcal{KS}(N)$ are non empty sets.

\begin{proposition}\label{carc0} Let $N$ be a squarefree composite number with prime divisors $p_{i}$, $1\leq i\leq m $. If we let

\begin{align*}
A_{i}=& \left\{ \dfrac{N+kp_{i}}{k+1}; \ \  k\in\mathbb{Z}\setminus\{-1\}\right\},
\end{align*}

for $1\leq i\leq m $, then

$$\mathbb{Q}\text{-}\mathcal{KS}(N)= \bigcap\limits_{\substack{1\leq i\leq m}}A_{i}.$$
\end{proposition}

\begin{proof}

Let $N\in\mathbb{N}\setminus\{0,1\}$ and $\alpha=\dfrac{\alpha_{1}}{\alpha_{2}}\in \mathbb{Q}\setminus\{0\}$.

By definition, we have $\alpha\in\mathbb{Q}$-$\mathcal{KS}(N)$ if and only if
$$\alpha_{2}p_{i}-\alpha_{1} \ \ \text{divides} \ \ \alpha_{2}N-\alpha_{1}\ \ \text{for each } i=1\ldots m. $$

 Equivalently, for each $i=1\ldots m$, there exist   $k_{i}\in\mathbb{Z}\setminus\{-1\}$   such that
 $N-\alpha=k_{i}(p_{i}-\alpha),$ namely $\alpha=\dfrac{N+(-k_{i})p_{i}}{(-k_{i})+1}\in A_{i}.$

\end{proof}

Throughout the rest of this paper and for a squarefree composite number $N$, we  set $M(k,p)=\dfrac{N+kp}{k+1}$ with $k\in\mathbb{Z}\setminus\{-1\}$ and $p$ is a prime divisor of $N$.

By Proposition~\ref{carc0}, we reprove in the following result, the property of the finitude  of $\mathbb{Q}$-$\mathcal{KS}(N)$ for each positive integer $N$.

\begin{theorem}\label{fini1}
For any given squarefree composite number $N$, there are only finitely many  $N$-Korselt rational bases.

\end{theorem}
\begin{proof}
We show that there exists a positive integer $k_{0}$ such that for all

 $k\in\mathbb{Z}\setminus\{-1\}$ and  $\mid k \mid>k_{0}$, we have $M(k,p)\notin \mathbb{Q}$-$\mathcal{KS}(N).$

Let $1\leq i\neq j \leq m$. If $\alpha\in \mathbb{Q}$-$\mathcal{KS}(N)$, then there exist integers $k_{i}$ and $k_{j}$  such that $\alpha=M(k_{i},p_{i})=M(k_{j},p_{j})$. Since $M(k_{i},p_{i})=\dfrac{N-p_{i}}{k_{i}+1}+p_{i}$, it follows that

 \begin{equation}\label{eq1}
  p_{j}-p_{i}=\dfrac{N-p_{i}}{k_{i}+1}-\dfrac{N-p_{j}}{k_{j}+1}
  \end{equation}
However, since $p_{j}-p_{i}$ is a fixed nonzero number and the limit of $ \dfrac{N-p_{i}}{k_{i}+1}-\dfrac{N-p_{j}}{k_{j}+1}$ as $k_{i}$ or $k_{j}$ approaches infinity is $0$, it follows by $\eqref{eq1}$, that the integers  $k_{i}$ and $k_{j}$ must be bounded; so that, there  exists a fixed positive integer $B_{(i,j)}$ such that $\max(\mid k_{i}\mid,\mid k_{j}\mid)<B_{(i,j)}$. Setting $B=\max_{1\leq i,j\leq m}B_{(i,j)}$, we get
$$\alpha\in\bigcap\limits_{\substack{p \, \text{prime}\\
                                  p\mid N}}
 \left\{ \dfrac{N+kp}{k+1}; \ \  \mid k\mid <B \right\},$$
it follows, by Proposition~\ref{carc0}, that  $$\mathbb{Q}\text{-}\mathcal{KS}(N)=\bigcap\limits_{\substack{p \, \text{prime}\\
                                  p\mid N}}
\left\{ \dfrac{N+kp}{k+1}; \ \  \mid k\mid <B \right\}.$$

Thus, $\mathbb{Q}\text{-}\mathcal{KS}(N)$ is finite.
\end{proof}

\bigskip
With a simple Maple program, we provide  data in Table $(1)$ ( resp. Table $(2)$) representing the rational Korselt set of  the five smallest squarefree composite numbers $N$  with two ( resp. three ) prime factors.

\begin{center}

\medskip
{\small
\setcellgapes{4pt}
\begin{tabular}{||l|l|l||l|l|l|l|l|} \hline

$N$ &$\mathbb{Q}$-$\mathcal{KS}(N)$ \\ \hline

 $2*3$& $\{4,\dfrac{3}{2} ,\dfrac{9}{4} , \dfrac{12}{5},\dfrac{18}{7},\dfrac{10}{3},\dfrac{5}{2},\dfrac{8}{3},\dfrac{14}{5}\}$ \\ \hline

 $2*5$&$\{4, 6,\dfrac{10}{3} ,\dfrac{5}{2},\dfrac{14}{3}\}$ \\ \hline

$2*7$& $\{8, 6, \dfrac{7}{2}\}$ \\ \hline

 $3*5$& $\{4,6,7,\dfrac{5}{2} ,\dfrac{10}{3} , \dfrac{25}{7},\dfrac{15}{4},\dfrac{45}{11},\dfrac{13}{3},\dfrac{9}{2},\dfrac{33}{7},\dfrac{27}{5},\dfrac{5}{3}\}$ \\ \hline

 $3*7$&$\{5, 6,9,\dfrac{15}{2} ,\dfrac{7}{3},\dfrac{7}{2},\dfrac{21}{5},\dfrac{21}{4},\dfrac{33}{5} \}$\\ \hline
\end{tabular}

\bigskip

\textbf{Table $(1)$.}  \emph{$\mathbb{Q}$-$\mathcal{KS}(N)$ for $N$ a squarefree composite number  with two prime factors}.
}\end{center}

\vspace{1cm}

\begin{center}

\medskip
{\small
\setcellgapes{4pt}
\begin{tabular}{||l|l|l||l|l|l|l|l|} \hline

$N$ &$\mathbb{Q}$-$\mathcal{KS}(N)$ \\ \hline

 $2*3*5$& $\{4, 6,\dfrac{15}{8},\dfrac{40}{13},\dfrac{5}{2},\dfrac{10}{3},\dfrac{15}{4},\dfrac{24}{5}\}$ \\ \hline
 $2*3*7$&$\{6,42,\dfrac{28}{9},\dfrac{9}{2},\dfrac{21}{8}\}$\\ \hline
$2*5*7$&$\{4,6,\dfrac{5}{2},\dfrac{7}{4},\dfrac{56}{11},\dfrac{25}{4},\dfrac{48}{7}\}$ \\ \hline
 $3*5*7$& $\{6,9,105,\dfrac{126}{25} ,\dfrac{35}{6} , \dfrac{90}{13},\dfrac{21}{5},\dfrac{35}{12}\}$\\ \hline
 $2*3*11$&$\{6,10,66,\dfrac{22}{7},\dfrac{11}{6}\}$\\ \hline
\end{tabular}

\bigskip

\textbf{Table $(2)$.} \emph{$\mathbb{Q}$-$\mathcal{KS}(N)$ for $N$ a squarefree composite number  with three prime factors}.
}\end{center}

\vspace{1cm}

The next result gives  bounds for all Korselt rational bases of a number $N$.
\begin{proposition}\label{carc1}
 The following assertions hold.
\begin{enumerate}
  \item $\gamma_{i}\leq M(j+r-i,p_{j})$ for each $(i,j)\in \{1\ldots r\}\times\{1\ldots m\}$.
  \item  $M(j-m-s-2,p_{j})\leq\beta_{s}$ for each $(s,j)\in \{1\ldots t\}\times\{1\ldots m\}$.

\end{enumerate}

\end{proposition}

\begin{proof}

1) First, Suppose that $N<\gamma_{r} $. Then $0<\gamma_{r}-N<\gamma_{r}-p_{m}$, and consequently $0<\lvert k_{m}\rvert=\dfrac{\gamma_{r}-N}{\gamma_{r}-p_{m}}<1$, contradicting $k_{m}\in \mathbb{Z}$.

 So, we deduce that  $\gamma_{1}<\gamma_{2}<\ldots<\gamma_{r} \leq N-1$.

Now, for given pair  $(i,j)\in \{1\ldots r\}\times\{1\ldots m\}$, let us prove that  $$\gamma_{i}\leq  \dfrac{N+(j+r-i)p_{j}}{j+r-i+1}.$$ Two cases are to be considered.
\begin{itemize}

\item If $\gamma_{i} \leq p_{j}$,  this is clear.

\item Assume that $p_{j}< \gamma_{i} < N$. Since $\gamma_{i}\in(\mathbb{Q_{+}})$-$\mathcal{KS}(N)$, then for each  $n=1\ldots j$, there exists an integer $f_{(i,n)} \in \mathbb{N}$ such that $N-\gamma_{i} =f_{(i,n)}(\gamma_{i}-p_{n})$.

Since, in addition, $( \gamma_{i}-p_{n})_{1\leq n\leq j}$ is a decreasing sequence, it follows that $( f_{(i,n)})_{1\leq n\leq j}$ is an increasing sequence.

 On the other hand, let $g_{(i,n)}=\dfrac{N-p_{n}}{\gamma_{i}-p_{n}}=f_{(i,n)}+1$.

 Since  $( \gamma_{i}-p_{n})_{1\leq i\leq r}$ is an increasing sequence, it follows that
 $(g_{(i,n)})_{1\leq i\leq r}$ and  $(f_{(i,n)})_{1\leq i\leq r}$ are two decreasing sequences.
  Consequently, as $f_{(r,1)}\geq 1$, we may write $$f_{(i,j)}> f_{(i,j-1)}>\ldots> f_{(i,1)} \geq f_{(i+1,1)}>\ldots> f_{(r,1)}\geq 1.$$

   Hence $\dfrac{N-\gamma_{i}}{\gamma_{i}-p_{j}}=f_{(i,j)}\geq j+r-i,$ which implies that

  $$\gamma_{i}\leq \dfrac{N+(j+r-i)p_{j}}{j+r-i+1}=M(j+r-i,p_{j}).$$
\end{itemize}

2) Of course, for each $(s,j)\in \{1\ldots t\}\times\{1\ldots m\}$, there exists  an integer $k_{(s,j)} \in \mathbb{N}$ such that $k_{(s,j)}=\dfrac{N-\beta_{s}}{p_{j}-\beta_{s}}$.

We claim that $k_{(1,m)}\geq 3$. Indeed, since $k_{(1,m)}=\dfrac{N-p_{m}}{p_{m}-\beta_{1}}+1$ and  $N>p_{m}$, it follows that $ k_{(1,m)}> 1$.

Next, we show that $k_{(1,m)}\neq2$. Suppose by contradiction that $k_{(1,m)}=2$. Then
$\beta_{1}= 2p_{m}-N\in\mathbb{Z}$, but as $\beta_{1}\neq p_{m}$ and $\beta_{1}\neq0$, we get  $N\neq
    p_{m}$ and $N\neq 2p_{m}$. So, there exists an integer $N_{1}\geq3$ such that  $N=N_{1}p_{m}$.

Let $p_{n}$ be a prime factor of $N_{1}$, then $$p_{n}-\beta_{1}= p_{n}+(N_{1}-2)p_{m} \mid N-\beta_{1}=2p_{m}(N_{1}-1).$$

As in addition $\gcd(p_{m},p_{n}+(N_{1}-2) p_{m} )=1$, it follows that
 $$p_{n}-\beta_{1}= p_{n}+(N_{1}-2)p_{m} \mid 2(N_{1}-1).$$

 Hence, $p_{n}+(N_{1}-2)p_{m} \leq 2(N_{1}-1)$. Since $4\leq p_{n}+2\leq p_{m}$, we get
 $$ 2+4(N_{1}-2)\leq p_{n}+(N_{1}-2)p_{m}\leq2(N_{1}-1).$$

 Thus $N_{1}\leq2$, which contradicts the fact that $N_{1}\geq3$. Consequently,  we conclude that $k_{(1,m)}\geq 3$.

 Now, on one hand, as $(p_{j}-\beta_{s})_{1\leq j\leq m}$ is an increasing sequence,  $(k_{(s,j)})_{1\leq j\leq m}$ is a decreasing sequence.

  On the other hand, setting $l_{(s,j)}=\dfrac{N-p_{j}}{p_{j}-\beta_{s}}=k_{(s,j)}-1$, it follows, since $(p_{j}-\beta_{s})_{1\leq s\leq t}$ is a decreasing sequence, that $(l_{(s,j)})_{1\leq s\leq t}$ and  $(k_{(s,j)})_{1\leq s\leq t}$ are two increasing sequences.  Knowing that $k_{(1,m)}\geq 3$, it follows that    $$k_{(s,j)}> k_{(s,j+1)}>\ldots> k_{(s,m)} > k_{(s-1,m)}>\ldots> k_{(1,m)}\geq 3.$$
   Hence $\dfrac{N-\beta_{s}}{p_{j}-\beta_{s}}=k_{(s,j)}\geq m+s-j+2$, which yields
   $$\beta_{s}\geq\dfrac{(m+s-j+2)p_{j}-N}{m+s-j+1}=M(j-m-s-2,p_{j}).$$

\end{proof}

Now, to prove the main result given by Theorem~\ref{encad}, we need to establish the following lemma.
\begin{lemma}\label{incdec}

The following assertions hold.
\begin{enumerate}
  \item $(M(2j,p_{2j}))_{j}$ and $(M(2j+1,p_{2j+1}))_{j}$ are two decreasing

  sequences.
  \item $(M(2j-m-3,p_{2j}))_{j}$ and $(M(2j-m-2,p_{2j+1}))_{j}$ are two

  decreasing sequences.

\end{enumerate}

\end{lemma}
\begin{proof}

First, as the result is immediate for $m=2$, we may suppose that $m\geq3$.

\begin{enumerate}
\item For $j\geq 1$, let

\medskip

\begin{tabular}{lll}
$\Delta_{j}$&=&$M(j,p_{j})-M(j+2,p_{j+2})$\\
&&\\

&=&$\dfrac{2N+j(j+3)p_{j}-(j+1)((j+2)p_{j+2}}{(j+3)(j+1)}$\\
&&\\
\end{tabular}

As $j+2\geq3$ and $p_{k}\geq k+1$ for each $k\geq1$, we get $$N\geq p_{j}p_{j+1}p_{j+2}\geq (j+1)(j+2)p_{j+2}.$$
Therefore
$\Delta_{j}\geq\dfrac{ N+j(j+3)p_{j}}{(j+3)(j+1)}>0$,  which implies that  $(M(2j,p_{j}))_{j}$ and $(M(2j+1,p_{2j+1}))_{j}$ are two decreasing sequences.

\bigskip

\item  For $j\leq m-2$, let

\medskip

\begin{tabular}{lll}
$\Gamma_{j}$&=&$M(j-m-3,p_{j})-M(j-m-1,p_{j+2})$\\
   &&\\
&=&$\dfrac{2N+(m-j)(m-j+3)p_{j}-(m-j+2)(m-j+1)p_{j+2}}{(m-j+2)(m-j)}$.\\
&&\\
\end{tabular}
For simplicity, let $a_{k}=m-j+k$. Then we can  write \begin{equation}\label{eq2}\Gamma_{j}=\dfrac{2N+a_{0}a_{3}p_{j}-a_{2}a_{1}p_{j+2}}{a_{0}a_{2}}.\end{equation}
Two cases are to be discussed.

\begin{itemize}

\item Suppose  that  $j\geq2$. Then $m\geq4$, hence $p_{j+2}\leq p_{m}$, $m+1-j\leq m-1\leq p_{m-2}$ and  $m+2-j\leq m \leq p_{m-1}$. This implies that
$$\begin{array}{lll}
a_{2}a_{1}p_{j+2}&=&(m+2-j)(m+1-j)p_{j+2}\\
&&\\
   &\leq &p_{m-1}p_{m-2}p_{m}\leq N. \\

   \end{array}$$

Therefore by $\eqref{eq2}$, we obtain  $$\Gamma_{j}> \dfrac{N+a_{0}a_{3}p_{j}}{a_{0}a_{2}}>0.$$

\item Now, assume that $j=1$. Then $m\geq 3$, and  we can write  $$N\geq p_{m-2}p_{m-1}p_{m}\geq (m-1)mp_{3}.$$

 Hence, by $\eqref{eq2}$ we get

$\begin{array}{lll}
&&\\
\Gamma_{j}&\geq&\dfrac{2(m-1)mp_{3}+(m-1)(m+2)p_{1}-m(m+1)p_{3}}{m^{2}-1}\\
&> &\dfrac{m(m-3)p_{3}}{m^{2}-1}\geq0.\\
&&
\end{array}$

\end{itemize}

So, we conclude that  $\Gamma_{j}>0$ for each $j=1\ldots m-2$.
Consequently, $(M(2j-m-3,p_{2j}))_{j}$ and $(M(2j-m-2,p_{2j+1}))_{j}$ are two decreasing sequences.
\end{enumerate}
\end{proof}

 By Proposition~\ref{carc1} and Lemma~\ref{incdec}, we get the following result which provides us with some information about the $\mathbb{Q}$-Korselt set of $N$.

\begin{theorem}\label{encad}

If $\alpha$ is an  $N$-Korselt rational base, then

$$M(-m-2,p_{1})\leq \alpha\leq \min(M(m-1,p_{m-1}),M(m,p_{m}))$$

\end{theorem}

\begin{proof} Let $\alpha \in\mathbb{Q}$-$\mathcal{KS}(N)$. Two case are to be considered.
\begin{enumerate}
  \item [\textbf{Case $1$}]: If  $\alpha>0$, then as by Proposition ~\ref{carc1}(1), we have $\gamma_{r}\leq M(j,p_{j})$ for each $ j=1\ldots m$, it follows  by Lemma ~\ref{incdec}, that $$\alpha\leq\gamma_{r}\leq \min_{1\leq j\leq m} M(j,p_{j})=\min\{ M(m-1,p_{m-1}),M(m,p_{m})\}.$$

 \item [\textbf{Case $2$}]:  Now, suppose that $\alpha<0$. Then  by Proposition ~\ref{carc1}(2), we get $M(j-m-3,p_{j})\leq\beta_{1}$ for each  $j=1\ldots m$. This implies, by Lemma $2.4$, that

$$\begin{array}{rll}
\max\{M(-m-2,p_{1}),M(-m-1,p_{2})\}&=& \max_{1\leq j\leq m} M(j-m-3,p_{j})\\
&&\\
&\leq & \beta_{1}.\\
\end{array}$$

We claim that $ M(-m-2,p_{1})>M(-m-1,p_{2})$. Indeed, let
\begin{align}
\Theta &=M(-m-2,p_{1})-M(-m-1,p_{2}) \nonumber \\
&=\dfrac{N+m(m+2)p_{1}-(m+1)^{2}p_{2}}{m(m+1)}\label{eq3}
\end{align}
 We consider two subcases.
\begin{enumerate}
  \item If $m\geq4$, then as $p_{k}\geq k+1$ for each $k\geq1$, we get $$N=p_{m}p_{m-1}\ldots p_{2}p_{1}\geq 2p_{m}p_{m-1}p_{2}\geq2(m+1)mp_{2}.$$
 It follows by  $\eqref{eq3}$, that

 $$\Theta>\dfrac{2(m+1)mp_{2}-(m+1)^{2}p_{2}}{m(m+1)}=\dfrac{(m-1)p_{2}}{m}>0.$$

  \item   If $m=3$. Then $$\Theta=\dfrac{p_{1}p_{2}p_{3}+15p_{1}-16p_{2}}{12}.$$
 This implies that
\begin{itemize}

\item If $p_{1}\geq3$ or $p_{3}\geq11$, then $p_{1}p_{3}\geq21$. So,

 $\Theta \geq\dfrac{5p_{2}+15p_{1}}{12}>0$.

\item If $(p_{1},p_{3})=(2,5)$ so that $p_{2}=3$, then $\Theta=1$.

\item If $(p_{1},p_{3})=(2,7)$ then $p_{2}\leq5$, and $$\Theta=\dfrac{14p_{2}-16p_{2}+30}{12}=\dfrac{-p_{2}+15}{6}>0.$$
\end{itemize}
\end{enumerate}
Therefore, in all cases  we obtain $\Theta>0$, which implies that $ M(-m-2,p_{1})>M(-m-1,p_{2})$.
Consequently, \\ $M(-m-2,p_{1})=\max\{ M(-m-2,p_{1}),M(-m-1,p_{2})\}\leq \beta_{1}.$
\end{enumerate}
 Finally, combining the two cases, we conclude that $$M(-m-2,p_{1})\leq \alpha\leq \min(M(m-1,p_{m-1}),M(m,p_{m})).$$
\end{proof}

\begin{example}\rm
By this example we show that the two values $M(m-1,p_{m-1})$ and $ M(m,p_{m})$ in the upper bound of the inequality of Theorem ~\ref{encad} can be reached.
\begin{itemize}
\item If $N=2*5$, then $m=2$ and we have
\begin{align*}
\min\{ M(m-1,p_{m-1}),M(m,p_{m})\}&=\min\{ M(1,p_{1}),M(2,p_{2})\} \\
&=M(1,p_{1})=6.
\end{align*}
\item If $N=3*5$, then $m=2$ and we have
\begin{align*}
\min\{ M(m-1,p_{m-1}),M(m,p_{m})\}&=\min\{ M(1,p_{1}),M(2,p_{2})\} \\
&=M(2,p_{2})=\dfrac{25}{3}.
\end{align*}
\end{itemize}
\end{example}

The next result gives a necessary and a sufficient condition for the upper  bound of the inequality in Theorem~\ref{encad} to be attained.
\begin{theorem}\label{attein}
The following assertions are equivalent.
\begin{enumerate}
  \item   $M(j,p_{j})\in\mathbb{Q}$-$\mathcal{KS}(N)$ for some $j\in\{1\ldots m\}.$

  \item  $N=2p_{2}$ $($i.e. $m=2$ and $p_{1}=2).$
\end{enumerate}

\end{theorem}

\begin{proof}
$(2)\Rightarrow (1)$

Suppose that $N=2p_{2}$ $($i.e. $m=2$ and $p_{1}=2)$. Let show that  $M(1,p_{1})\in\mathbb{Q}$-$\mathcal{KS}(N)$.

 Putting $\alpha=M(1,p_{1})=\dfrac{N+2}{2}=p_{2}+1$,  we get
  $2-\alpha =1- p_{2} $ divides $ N-\alpha=p_{2}-1$ and  clearly
  $p_{2}-\alpha= p_{2}-(p_{2}+1) = 1$ divides $N-\alpha$. Thus, $N$ is a $K_{\alpha}$-number and consequently $M(1,p_{1})\in\mathbb{Q}$-$\mathcal{KS}(N)$.

  \medskip

  $(1)\Rightarrow (2)$

Assume that $\alpha=M(j,p_{j})\in\mathbb{Q}$-$\mathcal{KS}(N)$. Then, by Proposition~\ref{carc0} and for each $k=1\ldots m$ such that $k\neq j$,  there exists an integer $l_{k} \in \mathbb{N}-\{j\}$ such that $\alpha=M(l_{k},p_{k})$. Therefore, by the equality $$\alpha=\dfrac{N+jp_{j}}{j+1}=\dfrac{N+l_{k}p_{k}}{l_{k}+1},$$ we get
\begin{equation}\label{eq4}(l_{k}-j)N=l_{k}(j+1)p_{k}-j(l_{k}+1)p_{j},\end{equation}
and so

\begin{equation}\label{eq5} N=(j+1)p_{k}+\dfrac{j((j+1)p_{k}-(l_{k}+1)p_{j})}{l_{k}-j}.\end{equation}

We consider two cases:

\begin{enumerate}
  \item  If $k>j$, then  since $p_{k}$ divides $j(l_{k}+1)$ and $j<k<p_{k}$, it yields by  $\eqref{eq4}$, that $p_{k}$ divides $l_{k}+1$.
    Hence $j+1\leq p_{j}< p_{k}\leq l_{k}+1$,  so that  $l_{k}-j>0 $ and  $(j+1)p_{k}-(l_{k}+1)p_{j}\leq0$. It follows by $\eqref{eq5}$, that   $N-(j+1)p_{k}\leq0$, that is $N\leq (j+1)p_{k}\leq p_{j} p_{k}$.

     But, as  $N\geq p_{j} p_{k}$, it follows that $N= (j+1)p_{k}= p_{j} p_{k}$. Hence $p_{j}=j+1$ so that $j=1, p_{j}=2$  and $ n=2$. Thus, $N= 2p_{2}.$

\item Suppose that $k<j$. Then $j\geq2$ and $p_{k}< p_{j}$.

Since $p_{j}\geq j+1$ and by  $\eqref{eq4}$, we have $p_{j}$ divides $l_{k}(j+1)$, we consider two subcases:
\begin{itemize}

\item If $p_{j}=j+1$, then $j=2$ and so $p_{2}=3$, $k=1$ and $p_{1}=2$.

Hence $\eqref{eq4}$ becomes $(l_{1}-2)N=3l_{1}p_{1}-2(l_{1}+1)p_{2}=-2p_{2}$. Thus, $l_{1}=1$ and $N=2p_{2}$.

\item Suppose that $p_{j}>j+1$. Then $p_{j}$ divides $l_{k}$, therefore $ j+1< p_{j}\leq l_{k}$.
Since in addition $p_{k}< p_{j}$, it follows that $l_{k}-j>0$ and $(j+1)p_{k}-(l_{k}+1)p_{j}<0$.   Thus by $\eqref{eq5}$, we obtain $N-(j+1)p_{k}\leq0,$ which gives $N\leq (j+1)p_{k}\leq p_{j} p_{k}$.

     But, as in addition  $N\geq p_{j} p_{k}$, we get $N= (j+1)p_{k}= p_{j} p_{k}$.
     Therefore $p_{j}=j+1$ and so $j=1, p_{j}=2$  and $ n=2$.
     Consequently $N= 2p_{2}.$

\end{itemize}

\end{enumerate}
So, the required equivalence holds.
\end{proof}

\noindent \textbf{Acknowledgement.}

\end{document}